\documentclass[letterpaper, 10 pt, conference]{ieeeconf}  

\IEEEoverridecommandlockouts                              
\usepackage{lipsum}
\usepackage{amsmath}
\usepackage{amssymb}

\usepackage{amsthm}
\usepackage{mathtools}
\usepackage{todonotes}

\usepackage{todonotes}

\theoremstyle{definition}
\newtheorem{definition}{Definition}
\newtheorem{theorem}{Theorem}
\newtheorem{lemma}{Lemma}
\newtheorem{assumption}{Assumption}
\newtheorem{problem}{Problem}

\newtheorem{remark}{Remark}

\usepackage{hyperref}
\hypersetup{
    colorlinks=true,
    linkcolor=black,
    filecolor=black,      
    urlcolor=blue,
    }
    
\usepackage{graphicx}
\graphicspath{ {resources/} }
\usepackage{booktabs}
\usepackage{float}
\usepackage{cite}

\let\labelindent\relax 
\usepackage{enumitem}

\title{\LARGE \bf
A Decentralized Control Framework for Energy-Optimal Goal Assignment and Trajectory Generation
}

\author{Logan E. Beaver, \textit{Student Member, IEEE}, Andreas A. Malikopoulos, \textit{Senior Member, IEEE}
	\thanks{This research was supported in part by ARPAE's NEXTCAR program under the award number DE-AR0000796 and by the Delaware Energy Institute (DEI).}
	\thanks{The authors are with the Department of Mechanical Engineering at the University of Delaware in Newark, DE 19716, USA
		{ (emails: \texttt{lebeaver@udel.edu}; \texttt{andreas@udel.edu)}} }%
}

\begin{document}

\maketitle
\thispagestyle{empty}

\begin{abstract}
	This paper proposes a decentralized approach for solving the problem of moving a swarm of agents into a desired formation. We propose a  decentralized assignment algorithm which prescribes goals to each agent using only local information. The assignment results are then used to generate energy-optimal trajectories for each agent which have guaranteed collision avoidance through safety constraints. We present the conditions for optimality and discuss the robustness of the solution. The efficacy of the proposed approach is validated through a numerical case study to characterize the framework's performance on a set of dynamic goals.
\end{abstract}

\section{INTRODUCTION}
\subsection{Motivation}
Complex systems are encountered in many applications, including cooperative autonomous agents, sensor fusion, and biological systems.
Referring to something as complex implies that it consists of interconnected agents which adapt and respond to their local and global environment.
As we move towards increasingly complex systems \cite{Malikopoulos2015},
new control approaches are needed to optimize the impact on system
behavior of the interaction between its entities \cite{Malikopoulos2015b,Malikopoulos}. 

Robotic swarm systems can exhibit complex behavior and have attracted considerable attention in many applications, e.g., transportation\cite{Malikopoulos2018, Malikopoulos2020a, Ren2007}, construction \cite{Lindsey2012, Joshi2014}, and surveillance\cite{Cort??s2009}. 
A common requirement for swarms is to move into a desired formation. However, due to cost constraints imposed on individual agents in a swarm, e.g., limited computation capabilities, battery capacity, and sensing abilities, any efficient control approach must take into account energy consumption. The task of moving in a specified formation has been explored in the literature  \cite{Oh2017,Oh2015,Brambilla2013}. However, achieving formations with minimum energy consumption during operation has not yet been thoroughly investigated. 


Several approaches to building cohesive formations in robotic systems have been proposed, such as formations built from triangular sub-structures \cite{Guo2010, Hanada2007}, where a scalable formation is achieved through the construction of a series of isosceles triangles. Methods inspired by crystal growth \cite{Song} and lattices structures \cite{Lee2008} have also shown promise. 
Other control methods using only scalar, bearing, or distance measurements were presented by Swartling et al. \cite{Swartling2014}. This approach was generalized to include the case where only a single leader agent was able to make distance or bearing measurements.

The problem of generating a desired formation was solved via scheduling by Turpin et al. \cite{Turpin}, where an initial assignment is achieved using a scheduling-based heuristic run on a central computer with global information. A significant amount of work, e.g., Wang and Xin \cite{Wang2013}, Sun and Cassandras \cite{Sun}, Xu and Carrillo \cite{Xu2015}, and Rajasree and Jisha \cite{Rajasree}, used optimization techniques in their solutions. However, these methods optimized the 
position of each agent in a virtual potential field and did not consider energy consumption by individual agents. 




The contribution of this paper is an assignment and trajectory generation algorithm which uses only local information for each agent. Other approaches, such as those by Turpin et al. \cite{Turpin2014}, Morgan et al. \cite{Morgan2016}, or Rubenstein et al. \cite{Rubenstein2012}, required global information in terms of a priori assignment, characteristics about the communication network size, or specifically oriented seed agents, respectively. 
Our proposed formulation is valid for any feasible initial and final conditions, requiring only that the initial and final positions be non-overlapping.
In addition, the formulation does not rely on potential fields \cite{Wang2013,Sun,Xu2015}, and instead produces energy-optimal trajectories which use proactive steering to avoid collisions.


The remainder of this paper proceeds as follows. In Section \ref{sec:problem}, we formulate the decentralized optimal control problem for each agent. In Section \ref{sec:solution}, we provide the problem formulation and solution approach of the assignment and trajectory generation and discuss implications on robustness. We present a numerical case study in Section \ref{sec:simulation}, which shows the behavior of the proposed method. Finally, we draw concluding remarks and discuss some ideas about future work in Section \ref{sec:conclusion}.


\section{PROBLEM FORMULATION} \label{sec:problem}

We consider the set $\mathcal{A} = \{1, \dots, N\},\, N \in \mathbb{N}_{>0},$ to index a system of autonomous agents in $\mathbb{R}^2$. The agents are moving into a desired formation indexed by a set of $\mathcal{F} = \{1, \dots, M\},\, M \in \mathbb{N}_{>0},$ goals. We consider the case where $N \leq M$, i.e., no redundant agents are brought to fill the formation, as shown in Fig. \ref{fig:introFigure}. This requirement can be relaxed by defining a behavior for excess agents, such as idling \cite{Turpin2014}. 

\begin{figure}[ht]
	\centering
	\includegraphics[width=0.75\linewidth]{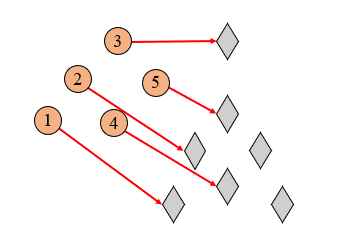}
	\caption{A group of $N=5$ agents entering a formation consisting of $M=7$ goals in $\mathbb{R}^2$.}
	\label{fig:introFigure}
\end{figure}

Each agent, $i\in\mathcal{A}$, is modeled as a double integrator 
\begin{align}
\mathbf{\dot{p}}_i(t) &=  \mathbf{v}_i(t), \label{eqn:pDynamics} \\
\mathbf{\dot{v}}_i(t) &=  \mathbf{u}_i(t),\label{eqn:vDynamics}
\end{align}
where $\mathbf{p}_i(t)\in\mathbb{R}^2$ and $\mathbf{v}_i(t)\in\mathbb{R}^2$ are the time-varying position and velocity vectors respectively, and $\mathbf{u}_i(t)\in\mathbb{R}^2$ is the control input over time $t\in[0, T_i]$, where $T_i\in \mathbb{R}_{>0},$ is the arrival time for agent $i$ to its assigned goal. Each agent's velocity and control input are bounded, namely,
\begin{align}
	v_{\min} \leq ||\mathbf{v}_i(t)|| \leq v_{\max},
	\label{eqn:vBounds}\\
	u_{\min} \leq ||\mathbf{u}_i(t)|| \leq u_{\max},
	\label{eqn:uBounds}
\end{align}
where $||\cdot||$ denotes the Euclidean norm, and $v_{\min}$, $v_{\max}$ and $u_{\min}$, $u_{\max}$ are the minimum and maximum allowable speed and control input respectively for each agent $i\in\mathcal{A}$. The state of each agent is the time-varying vector
\begin{equation}
\mathbf{x}_i(t) =
\begin{bmatrix}
\mathbf{p}_i(t) \\ \mathbf{v}_i(t)
\end{bmatrix}.
\end{equation}

Our objective is to develop a framework for the $N$ agents to optimally, in terms of energy, create any feasible formation of $M$ points while avoiding collisions between agents.
The energy consumption of each agent $i\in\mathcal{A}$ is given by
\begin{equation} \label{eqn:energy}
\dot{E}_i(t) = \frac{1}{2} ||\mathbf{u}_i(t))||^2.
\end{equation}
By minimizing the $L^2$ norm of the control input (acceleration/deceleration) we will have direct benefits in energy consumption.
\begin{definition} \label{def:goals}
	The \textit{desired formation} is the set of time-varying goals $\mathcal{G}(t) = \{\mathbf{p}_k(t) : \mathbb{R} \to \mathbb{R}^2 ~ | ~ k \in \mathcal{F}\}$. The set $\mathcal{G}$ can be prescribed offline, i.e., by a human  designer, or online by a high-level planner.
\end{definition}

Next, we present our modeling framework, which outlines the approach and assumptions used to solve the minimum energy desired formation problem.

\subsection{Modeling Framework}

In this framework, the agents can communicate with each other. The maximum sensing and communication range, $h\in\mathbb{R}_{>0}$, is used to define a neighborhood for each agent.

\begin{definition} \label{def:neighborhood}
	The \textit{neighborhood} of agent $i\in\mathcal{A}$ is defined as the time-varying set 
	\begin{equation*}
	\mathcal{N}_i(t) = \{j\in\mathcal{A} ~ | ~ \big|\big|\textbf{p}_i(t) - \textbf{p}_j(t) \big|\big| \leq h\}.
	\end{equation*}
\end{definition}

An agent $i\in\mathcal{A}$ is able to measure the relative position of any neighboring agent $j\in\mathcal{N}_i$. This leads to a natural definition of the scalar separating distance.

\begin{definition} \label{def:separating}
	The \textit{scalar separating distance} is defined as
	\begin{equation*}
	r_{ij}(t) = \big|\big|\mathbf{p}_i(t) - \mathbf{p}_j(t)\big|\big|.
	\end{equation*}
\end{definition}

Each agent $i\in\mathcal{A}$ occupies a closed disk of radius $R$. To guarantee no collisions between any agents $i,j\in\mathcal{A}$, we impose the following conditions on the system
\begin{align} \label{eqn:collisionCondition}
r_{ij}(t) &> 2R,\,\,  t\in\mathbb{R}_{>0}, \\
h &>> 2R. \label{eqn:sensingCondition}
\end{align}
To ensure each goal in the formation is feasible, the following condition should hold
\begin{equation} \label{eqn:formationSpacing}
    \min_{\mathbf{p}(t),\mathbf{q}(t) \in\mathcal{G}}\{||\mathbf{p}(t) - \mathbf{q}(t) ||\} <2R, ~  t\in\mathbb{R}_{>0}.
\end{equation}



In our modeling framework we impose the following assumptions:
\begin{assumption} \label{smp:perfect}
	The state $\mathbf{x}_i(t)$ for each agent $i\in\mathcal{A}$ is perfectly observed and there is negligible communication delay between the agents. 
\end{assumption}

Assumption \ref{smp:perfect} is required to evaluate the idealized deterministic performance of the generated optimal solution.

\begin{assumption} \label{smp:identical}
	All agents are homogeneous, and any agent may fill any goal in the formation.
\end{assumption}

This assumption simplifies the trajectory generation and assignment problems, and it can generally be relaxed by adding goal types as a constraint on the goal assignment.

\begin{assumption} \label{smp:energy}
	The energy cost of communication is negligible; the only energy consumption is in the form of (\ref{eqn:energy}). 
\end{assumption}

The strength of this assumption is application dependent. For cases with long-distance communications or high data rates, the trade-off for communication cost can be controlled by the selection of $h$.

Under this framework, the energy-optimal desired formation problem can be solved. This problem can be decomposed into two coupled subproblems: (1) goal assignment and (2) trajectory generation. Both of these problems are described in the following section, with emphasis on the goal assignment.

\section{Solution Approach} \label{sec:solution}
The decentralized desired formation problem is solved by decomposing it into the coupled goal assignment and trajectory generation subproblems. To decouple these problems the minimum energy objective in the assignment problem is approximated by the minimum Euclidean distance. Prior work, \cite{Turpin2014, Morgan2016}, has shown that this approximation is generally sufficient. This enables the assignment problem to be solved independently, which results in the endpoint constraints for the minimum-energy trajectory generation.

\subsection{Assignment Problem}
The objective of the assignment problem is to assign each agent to a goal such that the total distance traveled by all agents is minimized. In the decentralized case, each agent $i\in\mathcal{A}$ only has information about the positions of its neighbors, $j\in\mathcal{N}_i$, and the available goals, $\mathcal{G}$. A local assignment can be realized with the use of a local assignment matrix, $\mathbf{A}_i$,
\begin{equation}
	\begin{pmatrix}
		\mathbf{p}_1(T_1) \\ \mathbf{p}_2(T_2) \\ \vdots \\  \mathbf{p}_n(T_n)
	\end{pmatrix}
	=
	\mathbf{A}_i
	\begin{pmatrix}
		g_1 \\ g_2 \\ \vdots \\ g_M
	\end{pmatrix},
	\label{eqn:assignment}
\end{equation}
where $\mathbf{p}_j(T_j),~j\in\mathcal{N}_i$ are the final positions of each agent being assigned, $g_k, ~ k\in\mathcal{F},$ are the indices for each goal, and the elements $a_{jk}\in \mathbf{A}_i$ are binary assignment variables.
Each agent $i\in\mathcal{A}$ can solve (\ref{eqn:assignment}) independently as a linear program, and use the solution to select the prescribed goal.

\begin{definition} \label{def:prescribedGoal}
	For each agent $i\in\mathcal{A}$ the \textit{prescribed goal}, $\mathbf{p}^a_i(t)$, is defined as the goal assigned to agent $i$ for which
	\begin{equation}
		\mathbf{p}_i^a(t) \in \{\mathbf{p}_k \in \mathcal{G} ~|~ k\in\mathcal{F}, ~ a_{ik}=1,~ a_{ik} \in\mathbf{A}_i\},
	\end{equation}
	where the right hand side is a singleton set.
\end{definition}

It is possible for multiple agents to have the same prescribed goal. This occurs when two agents $i\in\mathcal{A},~ j\in\mathcal{N}_i,$ have different neighborhoods and use conflicting information to solve their assignment problem. This conflict is resolved by introducing the \textit{banned goal} set, defined next.
\begin{definition}\label{def:bannedGoal}
For any agent $i\in\mathcal{A}$, the \textit{banned goal} set is defined as the set $\mathcal{B}_i \subset \mathcal{G}$ which consists of all goals that agent $i$ is \emph{permanently banned from} when solving the goal assignment (\ref{eqn:assignment}).
\end{definition}

The following definitions and algorithm are presented for some agent $i\in\mathcal{A}$. However, all steps are performed \emph{simultaneously by all agents}.
For this agent $i$, the banned goal set is initially empty. Goals may be added to this set whenever 
the following condition is not satisfied
\begin{equation}
	\begin{aligned}
		\mathbf{p}^a_i(t) \neq \mathbf{p}^a_j(t),\,\, \forall j\in \mathcal{N}_i(t).
	\end{aligned}
	\label{eqn:NoConflicts}
\end{equation}

In the case that (\ref{eqn:NoConflicts}) is not satisfied, some agent(s) must be \emph{permanently banned} from the conflict goal, defined for agent $i$ as
\begin{equation} \label{eqn:conflictGoal}
    \mathbf{p}_c(t) \coloneqq \mathbf{p}_i^a(t).
\end{equation}
Banning is achieved by sequential application of ``tiebreaker" heuristics which compare:
\begin{enumerate}
    \item the size of each agent's neighborhood,
    \item the distance between each agent and the goal, and
    \item the index of each agent.
\end{enumerate}
Since the metrics of criteria 1, 2, and 3 are perfectly measurable (Assumption \ref{smp:perfect}), it follows that all agents must agree on the tiebreaker resolutions.
The tiebreaker hierarchy allows the banned goal set to be broken into three partitions,
\begin{equation}\label{eqn:banPartitions}
	\mathcal{B}_i(t) = \mathcal{B}^1_i(t) \cup \mathcal{B}^2_i(t) \cup \mathcal{B}^3_i(t),
\end{equation}
where superscripts $1$, $2$, and $3$ refer to the three tiebreakers, respectively. The tiebreakers are performed by all agents in the set of competing agents, defined next.
\begin{definition}\label{def:competingAgents}
	The set of \textit{competing agents} for agent $i\in\mathcal{A}$ is defined as
	\begin{equation*}
		\mathcal{C}_i(t) = \Big\{k\in\mathcal{N}_i(t) ~|~ \mathbf{p}_a^k(t) = \mathbf{p}_c(t) \Big\}.
	\end{equation*}
\end{definition}
When $|\mathcal{C}_i| > 1$ there are at least two agents, $i,j\in\mathcal{N}_i$ assigned to $\mathbf{p}_c$. Similarly to (\ref{eqn:banPartitions}), the set of competing agents can be split into three decreasing subsets,
\begin{align}
     \mathcal{C}_i^3 \subseteq \mathcal{C}_i^2 \subseteq \mathcal{C}_i^1 = \mathcal{C}_i
\end{align}
where the superscripts 1, 2, and 3 correspond to the agents which are comparing the three tiebreaker heuristics.


For each agent $i\in\mathcal{A}$, the banned goal sets partitions in \eqref{eqn:banPartitions} are defined as
\begin{align} \label{eqn:bannedUpdate}
\mathcal{B}^m_i(t) = &\Big\{\bigcup_{\tau=0}^t \Big(\{\mathbf{p}_i^a(\tau)\} \cap \Phi_i^m(\tau)\Big)\Big\},
\end{align}
where $\Phi^m_i(t)$ is given by
\begin{equation} \label{eq:indicators}
\Phi_i^m(t) =
\begin{cases}
\mathcal{G}, & \text{if } m=1, ~ i\neq \underset{j\in\mathcal{C}_i^1(t)}{\text{argmax}}\{|\mathcal{N}_j|(t)\}, \\
\mathcal{G}, & \text{if } m=2, ~ i\neq \underset{j\in\mathcal{C}_i^2(t)}{\text{argmax}}\{||\mathbf{p}_c(t)-\mathbf{p}_j(t)||\}, \\
\mathcal{G}, & \text{if } m=3, ~ i\neq \underset{j\in\mathcal{C}_i^3(t)}{\text{argmin}}\{j\}, \\

\emptyset, & \text{otherwise}.
\end{cases}
\end{equation}
where $m\in\{1, 2, 3\}$ again corresponds to the three tiebreaker heuristics.

To begin the tiebreaker process for agent $i\in\mathcal{A}$, consider the first conflict set $\mathcal{C}^1_i$ with the neighborhood heuristic. Every agent $j\in\mathcal{C}_i^1$ which satisfies
\begin{equation}\label{eqn:neighborGoalCondition}
	j = \arg\max_{k\in\mathcal{C}_i^1}\{|\mathcal{N}_k(t)|\},
\end{equation}
is eligible to be assigned to goal $\mathbf{p}_c$. If agent $i\in\mathcal{A}$ uniquely satisfies \eqref{eqn:neighborGoalCondition}, then the conflict test is complete and $i$ is assigned to $\mathbf{p}_c$. If $i$ does not satisfy \eqref{eqn:neighborGoalCondition}, then the goal $\mathbf{p}_c$ is added to $\mathcal{B}^1_i(t)$ as designated by \eqref{eq:indicators}. Finally, if agent $i$ does not uniquely satisfy \eqref{eqn:neighborGoalCondition} then the second criteria, distance to goal, must be compared. This comparison is done over a reduced conflict set,
\begin{equation}
	\mathcal{C}_i^2(t) = \Big\{j\in\mathcal{C}_i^1(t) ~|~ |\mathcal{N}_j(t)| = |\mathcal{N}_i(t)|  \Big\}.
\end{equation}

The second tiebreaker, maximum distance, is a minimax strategy which seeks to minimize the maximum distance traveled by any agent to the conflict goal. Again, every agent $j\in\mathcal{C}_i^2$ which satisfies
\begin{equation}\label{eqn:claimDistCondition}
	j = \underset{k\in\mathcal{C}_i^2}{\text{argmax}}\{||\mathbf{p}_k - \mathbf{p}_c||\},
\end{equation}
is eligible to be assigned to goal $\mathbf{p}_c$. If agent $i$ uniquely satisfies \eqref{eqn:claimDistCondition}, then the conflict test is complete and $i$ is assigned to $\mathbf{p}_c$. If $i$ does not satisfy \eqref{eqn:claimDistCondition}, then the goal $\mathbf{p}_c$ is added to $\mathcal{B}^2_i$ per \eqref{eq:indicators}. Finally, if $i$ satisfies \eqref{eqn:claimDistCondition}, but not uniquely, the final test must be taken over a further reduced conflict set, given by
\begin{equation}
	\mathcal{C}_i^3(t) = \Big\{j\in\mathcal{C}_i^2(t) ~|~ ||\mathbf{p}_c - \mathbf{p}_j|| = ||\mathbf{p}_c - \mathbf{p}_i||  \Big\},
\end{equation}
where the agent $k$ satisfying
\begin{equation} \label{eqn:idAgent}
k = \min\Big\{ j\in\mathcal{C}_i^3 \Big\},
\end{equation}
is assigned to the goal, and all other agents add $\mathbf{p}_c$ to $\mathcal{B}_i^3$ as designated by \eqref{eq:indicators}.
After the conflicts are resolved, if the size of $B_i$ has increased then the value of $T_i$ must also increase to
\begin{equation}
    T_i = t + T,
\end{equation}
where $t$ is the current time, and $T$ is a system parameter. This allows agent $i$ a sufficient amount of time to reach its new goal.
Finally, for each subsequent assignment involving agent $i\in\mathcal{A}$, when $\mathcal{B}_i(t) \neq \emptyset$ agent $i$ must broadcast its banned goal set to all $j\in\mathcal{N}_i$.

The assignment and banning process is iterated by all $j\in\mathcal{N}_i$ until (\ref{eqn:NoConflicts}) is satisfied in the entire neighborhood. The banned and restricted goal information is enforced through a constraint on the assignment problem, which follows.

\begin{problem}[Goal Assignment] \label{prb:assignment}
	Each agent assigns itself a goal independently by solving the linear minimum-distance assignment (\ref{eqn:assignment}).
	
	For each agent $i\in\mathcal{A}$, we have
	\begin{align}
	\underset{a_{jk}\in\mathbf{A}_i}{\text{min}} \Bigg\{
	\sum_{k\in\mathcal{N}_i} \sum_{j\in\mathcal{G}} a_{jk} \big|\big|\mathbf{p}_k(t) - \mathbf{p}_j^*(T_k)\big|\big| \Bigg\},\\
	 \mathbf{p}_k^0 \in \mathcal{N}_i,\, \mathbf{p}_j^*(t) \in \mathcal{G},\nonumber
	\end{align}
	subject to
	\begin{align}
			 \sum_{j\in\mathcal{G}} a_{jk} &= 1, ~~~ k\in\mathcal{N}_i, \label{eqn:p11} \\
			 \sum_{k\in\mathcal{N}_i} a_{jk} &\leq 1, ~~~ j\in\mathcal{G},\label{eqn:p12}\\
			 a_{jk} &= 0, ~~~ k\in\mathcal{N}_i, ~ \mathbf{p}_j\in\mathcal{B}_k,\label{eqn:p13} \\
			 a_{jk} &\in \{0, 1\}.\nonumber
	\end{align}
	
	Each agent independently solves Problem \ref{prb:assignment} as a linear program and selects its assigned goal. This process is repeated by each agent, $i\in\mathcal{A}$, until $|\mathcal{C}_i|=1$.
\end{problem}

As the safety constraints of Problem \ref{prb:assignment} explicitly depend on the neighborhood of agent $i\in\mathcal{A}$, the optimization must be recalculated each time the cardinality of the neighborhood of agent $i$ changes. Under weak assumptions about the trajectories of each agent, the assignments generated by Problem \ref{prb:assignment} is guaranteed to bring each agent to a unique goal as it is shown next.

\begin{lemma} \label{lma:solutionExistance}
	For every agent $i\in\mathcal{A}$, if $\big|\big(\bigcup_{j\in\mathcal{N}_i} \mathcal{B}_j\big) \setminus \mathcal{G}\big| \geq |\mathcal{N}_i|$, then the feasible region of Problem \ref{prb:assignment} is always nonempty.
\end{lemma}

\begin{proof}
    Let the set of goals available to all agents in the neighborhood of agent $i\in\mathcal{A}$ be denoted by the set
    \begin{equation} \label{eqn:feasibleGoals}
        \mathcal{V}_i(t) = \{\mathbf{p}\in\mathcal{G} ~|~ \mathbf{p}\not\in\mathcal{B}_j(t), ~ \forall j\in\mathcal{N}_i(t) \}.
    \end{equation}
    Let the injective function $m_i : \mathcal{N}_i(t) \to \mathcal{V}_i(t)$ map each agent to a goal. As $|\mathcal{N}_i| \leq |\mathcal{V}_i(t)|$, the function $m_i$ must always exist and imposes a mapping from each agent to a unique goal.
    
    Since $m_i$ is injective, it satisfies (\ref{eqn:p11}) and (\ref{eqn:p12}). Likewise, $\mathcal{V}_i\subset\mathcal{B}_j^c$ for all $j\in\mathcal{N}_i$, and therefore the imposes mapping satisfies (\ref{eqn:p13}). Therefore, the mapping imposed by the function $m_i$ is a feasible solution to Problem \ref{prb:assignment}.
    %
	%
	%
\end{proof}

For a sufficiently large value of $T$, the convergence of all agents to goals is guaranteed by Theorem  \ref{thm:assignmentConvergence}.

\begin{theorem}[Assignment Convergence] \label{thm:assignmentConvergence}
	Under the assumptions of Lemma \ref{lma:solutionExistance}, for a sufficiently large value of $T$, and if the energy-optimal trajectories for each agent always move toward their assigned goal, then 
	all $i\in\mathcal{A}$ must reach an assigned goal in finite time.
\end{theorem}

\begin{proof}
    Let $\{g_n\}_{n\in\mathbb{N}}$ be the sequence of goals assigned to agent $i\in\mathcal{A}$ as designated by the solution of Problem \ref{prb:assignment}.  From Lemma \ref{lma:solutionExistance}, $\{g_n\}_{n\in\mathbb{N}}$ is not empty, and the elements of this sequence are integers bounded by $1\leq g_n \leq |\max{\mathcal{F}}|$. Thus, the range of this sequence is compact and must be (1) finite, ($2$) convergent, or ($3$) periodic.
    
    ($1$) For a finite sequence, $T_i$ is bounded by $T \cdot |\mathcal{G}|$.
    
    ($2$) Under the discrete metric, an infinite convergent sequence requires that there exists $N\in\mathbb{N}_{>0}$ such that $g_n = p$ for all $n>N$ for some formation index $p\in\mathcal{F}$. This reduces to case $1$, as $T_i$ does not increase for repeated assignments to the same goal.
    
    ($3$) By the Bolzano-Weierstrass Theorem, an infinite non-convergent sequence $\{g_n\}_{n\in\mathbb{N}}$ must have a convergent subsequence, i.e., agent $i$ is assigned to some subset of goals $\mathcal{I}\subseteq\mathcal{G}$ infinitely many times with some number of intermediate assignments for each goal $\mathbf{g}\in\mathcal{I}$. From the construction of the banned goal set, we must have $\mathcal{I}\bigcap\mathcal{B}_i(t) = \emptyset$ for all $t\in[0,T_i]$. This implies that, by the update method of $T_i$, the position of all goals, $g(t)\in\mathcal{I}$ must only be considered at time $T_i$, which we denote as $\mathbf{g}(T_i)\in\mathcal{I} = \mathbf{g}\in\mathcal{I}$.
    
    This implies that the goals available to agent $i$, i.e., $\mathcal{I} = \mathcal{G}\setminus\mathcal{B}_i$, must be shared between $n>0$ other periodic agents. This implies at some time $t_1$ that a goal, $\mathbf{g}\in\mathcal{I}$, must be an optimal assignment for agent $i$, a non optimal assignment at time $t_2>t_1$ and an optimal assignment at time $t_3>t_2$. This implies the distance between agent $i$ and goal $\mathbf{g}$ satisfy
    \begin{align}
        |\mathbf{p}_i(t_1) - \mathbf{g}| &< |\mathbf{p}_i(t_1) - \mathbf{g'}|, \\
        |\mathbf{p}_i(t_2) - \mathbf{g'}| &< |\mathbf{p}_i(t_2) - \mathbf{g}|, \\
        |\mathbf{p}_i(t_3) - \mathbf{g}| &< |\mathbf{p}_i(t_3) - \mathbf{g'}|,
    \end{align}
    for some goal $g'\in\mathcal{I}, ~ g'\neq g$. Agent $i$ must not increase his distance from his assigned goal, which implies
    \begin{align}
        |\mathbf{p}_i(t_1) - \mathbf{g}| &> |\mathbf{p}_i(t_2) - \mathbf{g}|, \\
        |\mathbf{p}_i(t_2) - \mathbf{g'}| &> |\mathbf{p}_i(t_3) - \mathbf{g'}|,\\
    \end{align}
    and hence
    \begin{align}
        |\mathbf{p}_i(t_1) - \mathbf{g'}| & > |\mathbf{p}_i(t_3) - \mathbf{g'}|,
    \end{align}
    which is satisfied for all goals $g'\in\mathcal{I}$. This is only possible if agent $i$ simultaneously approaches all goals in $\mathcal{I}$, which implies they are arbitrarily close. This contradicts \eqref{eqn:formationSpacing}, and thus no such periodic behavior may exist.

\end{proof}

\subsection{Optimal Trajectory Generation}

After the goal assignment is determined, each agent must generate a collision-free and energy-optimal trajectory to their assigned goal. 
The initial and final condition constraints for any agent $i\in\mathcal{A}$ are given by
\begin{align}
	\Big(\mathbf{p}_i(t_0) - \mathbf{p}_{i,0},\, \mathbf{v}_i(t_0) - \mathbf{v}_{i,0}\Big) &= \Big(\mathbf{0},\, \mathbf{0}\Big),
	\label{eqn:ICs}\\
	\Big(\mathbf{p}_i(T_i) - \mathbf{p}_{i,f},\, \mathbf{v}_i(T_i) - \mathbf{v}_{i,f}\Big) &= \Big(\mathbf{0},\, \mathbf{0}\Big),
	\label{eqn:BCs}
\end{align}
where 
$\mathbf{p}_{i,f}, \mathbf{v}_{i,f}$ are the result of solving Problem \ref{prb:assignment}.
To resolve the coupling introduced by collision avoidance, each agent $i\in\mathcal{A}$ predicts the optimal trajectories of its neighbors, $j\in\mathcal{N}_i$ to select its prescribed trajectory.
\begin{definition}
	The \textit{prescribed trajectory}, $u_i^a(t)$, is the trajectory assigned to agent $i$ after solving for the optimal trajectories of every agent in its neighborhood, $j\in\mathcal{N}_i$.
\end{definition}
For agent $i$ to calculate its prescribed trajectory, $u_i^a(t)$, the trajectory optimization problem must be solved over the set
\begin{equation*}
	\mathbf{U}_i(t) = \big\{u_j(t) : j\in\mathcal{N}_i\big\},
\end{equation*}
such that
\begin{equation*}
	u_i^a(t) = u_i(t) \in \mathbf{U}_i(t).
\end{equation*}
This can be achieved by the quadratic optimization problem given by:

\begin{problem}[Trajectory Generation] \label{prb:trajectory}
	For each agent $i\in\mathcal{A}$, we have
		\begin{align}
			\underset{\mathbf{u}_i\in\mathbf{U}}{\text{min}} \Bigg\{
			\sum_{j\in\mathcal{N}_i} \int_{\tau=t}^{T_i} ||\mathbf{u}_j(\tau)|| \text{d}t \Bigg\},
		\end{align}
		subject to
		\begin{align*}
			\text{Dynamic constraints }&(\ref{eqn:pDynamics}), (\ref{eqn:vDynamics}), \\
			\text{State and control bounds }&(\ref{eqn:vBounds}), (\ref{eqn:uBounds}), \\
			\text{Collision avoidance }&(\ref{eqn:collisionCondition}),\\
			\text{Boundary constraints }&(\ref{eqn:ICs}), (\ref{eqn:BCs}).\\
		\end{align*}
\end{problem}

Problem \ref{prb:trajectory} can be solved as an iterated quadratic program with a similar conflict framework as Problem \ref{prb:assignment}, where one agent fixes its trajectory and others steer to avoid it. 

Problems \ref{prb:assignment} and \ref{prb:trajectory} are solved sequentially at time $t=0$ to achieve an initial set of assignments and corresponding optimal trajectories. As both optimizations only use local information, each agent must resolve each problem whenever their neighborhood changes. This ensures that every agent is using all available information to optimize their trajectories while guaranteeing collision avoidance.

\begin{remark}\label{rmk:centralized}
	The solutions of Problems \ref{prb:assignment} and \ref{prb:trajectory} reduce to the centralized case as $h \to \infty.$	
\end{remark}
Remark \ref{rmk:centralized} relies on the fact that, as $h \to\infty$, it must be true that $\mathcal{N}_i(t_0) = \mathcal{A}$ and $\mathcal{B}_i(t_0) = \emptyset$   $\forall i\in\mathcal{A}$.
Hence, problems \ref{prb:assignment} and \ref{prb:trajectory} simply reduce to each agent solving the centralized problem individually.


\section{Simulation Case Study}\label{sec:simulation}
 
 To give insight into the behavior of the agents a series of simulations were performed in Matlab. Each simulation lasted for $20$ s or until all agents reach their assigned goal, whichever was longer. The centroid of the formation moved with a fixed velocity, while the leftmost and rightmost three goals included additional periodic motion
 
 The minimum separating distance between agents, total energy consumed, and maximum velocity for the unconstrained solutions to Problem \ref{prb:trajectory} are given as a function of the horizon in Table \ref{tab:resultsEnergy}. A graph of each agent's position over time for two cases is given in Fig. \ref{fig:trajectory1} and \ref{fig:trajectory3}.

\begin{figure}[ht]
    \centering
    \includegraphics[width=0.8\columnwidth]{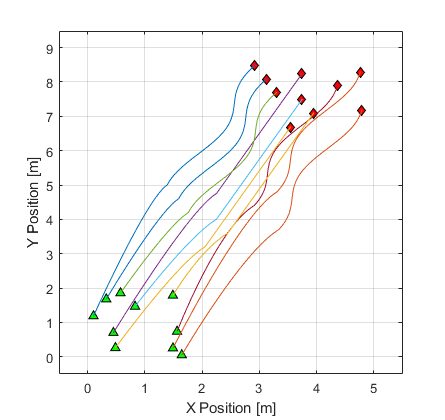}
    \caption{Agent trajectories for the centralized case, here the assignment globally minimizes distance travelled and trajectories are evenly spaced.}
    \label{fig:trajectory1}
\end{figure}
\begin{figure}[ht]
    \centering
    \includegraphics[width=0.9\columnwidth]{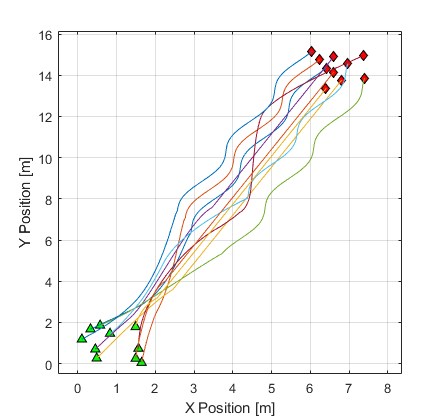}
    \caption{Agent trajectories for $R=0.5$m, one grid cell in diameter. Agents use very little information and tend to group at the nearest goal position.}
    \label{fig:trajectory3}
\end{figure}

\begin{table}[h]
    \centering
    \begin{tabular}{cccc}
     $h$ [m] & min. separation [cm]& $E$ [kJ/kg] & $t_f$ [s] \\\toprule
     $ \inf $ & $16.5$ & $63.77$ & $20$\\
     $ 1.60 $ & $0.82$ & $83.76$ & $27$\\
     $ 1.50 $ & $1.21$ & $56.43$ & $20$\\
     $ 1.40 $ & $0.38$ & $140.7$ & $41$\\
     $ 1.30 $ & $5.25$ & $52.26$ & $20$\\
     $ 1.10 $ & $0.32$ & $96.13$ & $34$\\
     $ 0.95 $ & $0.54$ & $41.61$ & $20$\\
     $ 0.75 $ & $0.60$ & $227.7$ & $42$ \\
     $ 0.50 $ & $2.41$ & $140.1$ & $39$
    \end{tabular}
    \caption{Numerical results for N=M=10 agents and goals, with a time parameter of $T=10$ s and various sensing distances.}
    \label{tab:resultsEnergy}
\end{table}

The results in Table \ref{tab:resultsEnergy} generally show no correlation between energy consumption and sensing horizon. In fact, the minimum energy consumption occurs near $R=1.3$ m rather than the centralized case. This is likely a result of the minimum distance approximation, which does not account for the required change in velocity for a dynamic formation with moving goals.


\section{Conclusion}\label{sec:conclusion}

In this paper, we proposed an approach for solving the desired formation problem of a group of autonomous agents. We presented a  formulation of the formation reconfiguration problem and introduced a concept of prescribed goals and trajectories. The robustness and convergence properties of the system were discussed, and the performance was characterized relative to the centralized approach. A numerical solution was presented for $N=M=10$ agents and goals, and the system performance metrics were compared relative to the sensing radius.

Future areas of research include: relaxing the assumptions on Lemma \ref{lma:solutionExistance} to characterize when solutions exist, incorporating information from outside the neighborhood into goal assignment, analyzing the effect on sensing radius on communication cost versus convergence and propulsion energy, reducing the computational load in calculating Problem \ref{prb:trajectory}, and characterizing the optimality of the tiebreaker heuristics.

\bibliographystyle{IEEEtran}
\bibliography{BibliographyFull.bib}

\end{document}